\documentclass{amsart}

\usepackage{graphicx}
\usepackage{amsthm}
\usepackage{amsmath}
\usepackage{amsfonts}
\usepackage{amssymb}
\usepackage[usenames,dvipsnames]{color}
\usepackage{tikz-cd}

\usepackage[colorlinks=true,citecolor=blue,urlcolor=blue,%
linkcolor=blue,pdfpagemode=UseNone]{hyperref}

\graphicspath{{./}{figures/}}

\theoremstyle{plain}
\newtheorem{theorem}{Theorem}[section]
\newtheorem{lemma}[theorem]{Lemma}
\newtheorem{corollary}[theorem]{Corollary}
\newtheorem{proposition}[theorem]{Proposition}
\theoremstyle{definition}

\newtheorem{definition}[theorem]{Definition}
\theoremstyle{remark}
\newtheorem{remark}[theorem]{Remark}
\newtheorem{example}[theorem]{Example}

\numberwithin{equation}{section}

\newcommand{\cA}{{\mathcal A}}
\newcommand{\cB}{{\mathcal B}}
\newcommand{\cF}{{\mathcal F}}
\newcommand{\cG}{{\mathcal G}}

\newcommand{\cU}{{\mathcal U}}
\newcommand{\cV}{{\mathcal V}}

\newcommand{\cX}{{\mathcal X}}
\newcommand{\cY}{{\mathcal Y}}

\newcommand{\bR}{{\mathbb R}}

\DeclareMathOperator{\im}{im}
\DeclareMathOperator{\interior}{int}

\newcommand{\mvmap}{{\,\overrightarrow{\to}\,}}

\begin{document}

\title[Inducing a map on homology from a correspondence]%
{Inducing a map on homology \\ from a correspondence}

\author[S.~Harker]{Shaun Harker}
\address{SH: Department of Mathematics,
Hill Center-Busch Campus, Rutgers, The State University of New Jersey,
110 Frelinghusen Rd, Piscataway, NJ 08854-8019, USA}

\author[H.~Kokubu]{Hiroshi Kokubu}
\address{HK: Department of Mathematics\ /\ JST CREST, Kyoto University,
Kyoto 606-8502, Japan}

\author[K.~Mischaikow]{Konstantin Mischaikow}
\address{KM: Department of Mathematics and BioMaPS,
Hill Center-Busch Campus, Rutgers, The State University of New Jersey,
110 Frelinghusen Rd, Piscataway, NJ 08854-8019, USA}

\author[P.~Pilarczyk]{Pawe\l{} Pilarczyk}
\address{PP: Institute of Science and Technology Austria, Am Campus 1,
3400 Klosterneuburg, Austria}


\date{November 18, 2014}

\begin{abstract}
We study the homomorphism induced in homology
by a closed correspondence between topological spaces,
using projections from the graph of the correspondence
to its domain and codomain.
We provide assumptions under which the homomorphism
induced by an outer approximation of a continuous map
coincides with the homomorphism induced in homology by the map.
In contrast to more classical results we do not require
that the projection to the domain have acyclic preimages.
Moreover, we show that it is possible to retrieve
correct homological information from a correspondence
even if some data is missing or perturbed.
Finally, we describe an application to combinatorial maps
that are either outer approximations of continuous maps
or reconstructions of such maps from a finite set of data points.
\end{abstract}

\maketitle


\section{Introduction}
\label{sec:intro}

The focus of this paper is on effective methods
for computing $f_*\colon H_*(X,A)\to H_*(Y,B)$,
the homomorphism induced on homology by a continuous map
$f \colon (X, A) \to (Y, B)$, in situations where the map $f$
and/or the underlying spaces are only known via finite approximations.
This problem arises naturally in the context of topological data analysis
or the use of algebraic topological invariants to study nonlinear dynamics.
In general, in these settings, at best one has bounds for the action of $f$,
though estimates of questionable certainty are more likely.
With this in mind, we use correspondences
to represent $f$ (see Section~\ref{sec:prelim} for definitions)
and complexes to represent the domain and codomain.

To put the goals of this paper into perspective,
consider the simpler case of $f\colon X\to Y$.
Let $F\subset X\times Y$ be a correspondence.
Let $p\colon F\to X$ and $q\colon F\to Y$ be the canonical projection maps.
Assume
\begin{description}
\item[A1] $p(F)=X$,
\item[A2] $f(x)\in q(p^{-1}(x))$, for all $x\in X$, and 
\item[A3] $p^{-1}(x)$ is acyclic for all $x\in X$.
\end{description}
Recall the following classical result of Vietoris \cite{vietoris}.
\begin{theorem}
\label{thm:mapping}
Let $Z$ and $X$ be compact metric spaces.
Let $p \colon Z \to X$ be a surjective continuous function.
If $p^{-1} (x)$ is acyclic for every $x \in X$
then $p_*\colon H_*(Z)\to H_*(X)$ is an isomorphism.
\end{theorem}
In the context of this paper, an immediate consequence
is that $f_* = q_*\circ (p_*)^{-1}$.
On a theoretical level the computation of $f_*$
in terms of acyclic correspondences has a long tradition
dating back at least to the work
of Eilenberg and Montgomery~\cite{eilenberg:montgomery}.
To the best of our knowledge, the first explicit algorithms
based on this idea are presented in \cite{MMP05},
where it is implemented in the context of cubical complexes.
An alternative approach, based on discrete Morse theory and applicable
to arbitrary complexes, is presented in \cite{HMMN13}.
Code based on these implementations can be found at \cite{chomp}.

From the perspective of applications, all three assumptions
{\bf A1}--{\bf A3} are problematic.
In the context of data analysis, the domain used to compute $f_*$
is a complex $\mathsf X$ constructed from a finite data set
and the true domain $X$ of $f$ is unknown.
Thus, as is demonstrated in Example~\ref{ex:henon},
there are situations in which the appropriate assumption
is that $\mathsf X\subset X$, in which case ${\bf A1}$ fails. 

Again, in the context of data analysis one expects
that the data representing $f$ is corrupted by bounded noise,
but in many situations one does not have a clear estimate
on the size of the noise.
Thus, as indicated in Example~\ref{example:noisysample},
the correspondence $F$ constructed using the data
may fail to satisfy {\bf A2}.

Finally, even in settings where complete knowledge of the function $f$
can be assumed, the assumption of acyclicity may be too strong.
Consider a smooth function $f\colon \bR^n\to \bR^n$
and assume that we are interested in computing $f_*$
restricted to a compact subset $X$ represented in terms of a cubical complex,
where $X$ is a set defined in terms of the action induced by $f$.
Problems of this type occur naturally in nonlinear dynamics,
where $(X,A)$ might represent an index pair and $f_*$
becomes a representative for the associated Conley index 
(see \cite{arai:kalies:kokubu:mischaikow:oka:pilarczyk,
chaos,bush:mischaikow,mischaikow:mrozek}
for more details and concrete examples).

Using the smoothness, there are a variety of methods
by which rigorous bounds on the images of $n$-dimensional cubes
can be computed \cite{capd}.
It is natural to construct a correspondence as follows.
For each $n$-dimensional cube, define the correspondence
in terms of the set of cubes that cover the outer approximation
of the image, based on the rigorous bounds (see Figure~\ref{fig:mapWrap}).
To define the correspondence at the intersection of two $n$-dimensional cubes
in such a way that the compactness assumption of Theorem~\ref{thm:mapping}
is satisfied, it is natural to declare the correspondence
in terms of the union of the correspondences over the $n$-dimensional cubes.
However, since in general the union of acyclic sets need not be acyclic,
at this point one can no longer assume that {\bf A3} is satisfied.

\begin{figure}[htbp]
\def\svgwidth{0.75\textwidth}
\centerline{\input{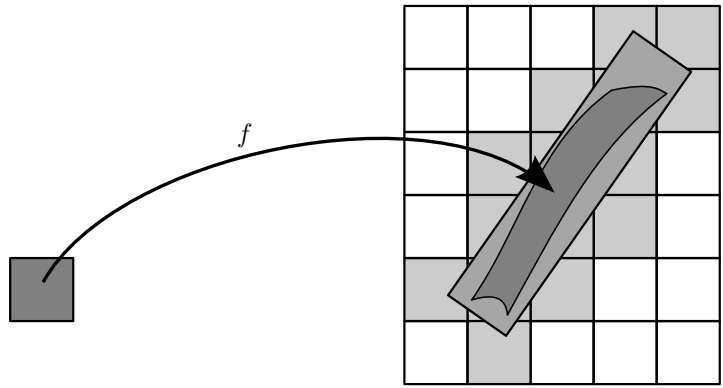}}
\caption{%
Suppose that an outer bound for the image (shaded in dark gray)
of the $2$-dimensional cube shown in the left is given
in terms of a rectangular set (shaded in medium gray)
that is not aligned with the grid.
If one covers it with the minimal collection of grid elements
that intersect the rectangular set (all shaded in bright gray)
then the obtained set is acyclic, but not convex.
The minimal convex covering contains twice the number of cubes
and thus is a much poorer approximation of the dynamics induced by~$f$.}
\label{fig:mapWrap}
\end{figure}

Keeping these examples in mind, we provide an outline for this paper.
After providing necessary preliminary definitions
in Section~\ref{sec:prelim}, we introduce in Section~\ref{sec:main}
the concepts of {\em homological completeness}
and {\em homological consistency} which allow us to characterize
whether the correspondence contains sufficient information
to recover the appropriate homological information
about the domain and the image into the range, respectively.
In particular, this leads to our main result Theorem~\ref{thm:homF}
in which we guarantee the correct computation of
$f_*\colon H_*(X,A)\to H_*(Y,B)$ based on correspondences
that do not necessarily satisfy {\bf A1}--{\bf A3}.
To extend the applicability of this result,
in Section~\ref{sec:extensions} we focuses
on relating homological information between different correspondences.

As is suggested above, our results have an application
in the context of writing software for computational homology.
For example, to apply \cite{HMMN13,MMP05} it is not only necessary
to obtain an outer approximation of the true underlying function $f$
but also to ensure that the outer approximation is acyclic-valued.
In order to guarantee this, the associated software relies
on performing homology computation using uniform grids
and grid-aligned bounding rectangles. As suggested
in Figure~\ref{fig:mapWrap}, this results in undesirable overestimates.
The results of this paper relax the requirement
to homological completeness and consistency,
which permits algorithms that are more robust,
allowing for a wider choice of complexes and tighter outer approximations.
To emphasize this, we conclude in Section~\ref{sec:comb}
with examples that specifically highlight the relevance
of homological completeness and consistency.


\section{Preliminaries.}
\label{sec:prelim}
Let $X$ and $Y$ be topological spaces.
Let $A \subset X$ and $B \subset Y$.
A \emph{correspondence} from $(X, A)$ to $(Y, B)$
is a pair $(F, F')$ of relations
such that $F \subset X \times Y$ and $F' \subset A \times B$,
and $F' \subset F$.
We identify each relation with its graph.
In particular, we consider the following special correspondences:
a \emph{map}, where every $x \in X$
corresponds to exactly one $y \in Y$ through $F$,
and every $a \in A$ corresponds to exactly one $b \in B$ through $F'$;
and a \emph{multivalued map}, where every $x \in X$
corresponds to at least one $y \in Y$ through~$F$,
and every $a \in A$
corresponds to at least one $b \in B$ through $F'$.
If $(F, F')$ is a map from $(X, A)$ to $(Y, B)$ then $F'$
is uniquely determined by $F$ and $A$, and obviously $F (A) \subset B$,
which justifies the usual notation $F \colon (X, A) \to (Y, B)$.
We remark that requiring that $F (A) \subset B$
in the case of general correspondences is too restrictive
(see Section~\ref{sec:comb} and~\cite{MMP05});
therefore, working explicitly with pairs of relations is necessary.

A correspondence $(F, F')$ is \emph{closed}
if both $F$ and $F'$ are closed as subsets of $X \times Y$.
A correspondence $(G, G')$ is an \emph{enlargement} of $(F, F')$
if $F \subset G$ and $F' \subset G'$,
and is denoted by $(F, F') \subset (G, G')$.
Note that an enlargement of a [multivalued] map is a multivalued map.
A map $f \colon (X, A) \to (Y, B)$ is a \emph{selector}
of a correspondence $(F, F')$ from $(X, A)$ to $(Y, B)$
if $(F, F')$ is an enlargement of $f$.
Observe that if a correspondence $(F, F')$
has a selector then $(F, F')$ is a multivalued map.

A relation $F \subset X \times Y$ is \emph{upper semicontinuous}
if for every $x \in X$, and for every neighborhood $V$ of $F (x)$,
there exists a neighborhood $U$ of $x$ such that $F (U) \subset V$.
A correspondence $(F, F')$ is \emph{upper semicontinuous}
if both $F$ and $F'$ are upper semicontinuous.
Recall the following classical result \cite{browder}:
\begin{theorem}
\label{thm:closed}
Let $F \subset X \times Y$ be a relation.
If $F (x)$ is closed for every $x \in X$,
the domain of $F$ is closed, and the range $Y$ of $F$ is compact,
then $F$ is closed (as a subset of $X \times Y$)
if and only if $F$ is upper semicontinuous.
\end{theorem}

We consider homology with coefficients in an arbitrary ring $R$,
which will be fixed throughout the paper.
We leave the freedom to choose the most appropriate
homology theory to work with for the reader; however, in the examples
we shall use cubical homology \cite{KMM04}.

A set is \emph{acyclic} with respect to $R$ if its homology is isomorphic
with the homology of a single point,
that is, $H_0 \cong R$ and $H_q \cong \{0\}$
for all $q \neq 0$.
A correspondence $(F, F')$ is \emph{acyclic}
if the image of every point by $F$ is an acyclic set,
and also by $F'$ whenever applicable.

\begin{definition}
\label{def:homF}
Let $(F, F')$ be a closed correspondence from $(X, A)$ to $(Y, B)$.
Let $p \colon (F, F') \to (X, A)$ and $q \colon (F, F') \to (Y, B)$
be the canonical projections.
Let $\widetilde{p_*}$ denote the isomorphism induced by $p_*$
between the quotient space $H_* (F, F') / \ker p_*$ and $\im p_*$.
Let $\widetilde{q_*}$ denote the homomorphism induced by $q_*$
between the quotient space $H_* (F, F') / \ker p_*$
and $H_* (Y, B) / q_* (\ker p_*)$.
The \emph{homomorphism $(F, F')_*$ induced in homology by $(F, F')$}
is defined by
\[
(F, F')_* := \widetilde{q_*} \circ \widetilde{p_*}^{-1} \colon
\im p_* \to H_* (Y, B) / q_* (\ker p_*).
\]
\end{definition}

\begin{remark}
\label{rem:homFcont}
If $f\colon X\to Y$ is a continuous map then $p_*$ is an isomorphism,
and the homomorphism obtained by Definition~\ref{def:homF} applied to $f$
coincides with the homomorphism induced in homology
by the continuous map $f$ in the usual sense.
\end{remark}

\begin{figure}[htbp]
\def\svgwidth{0.6\textwidth}
\centerline{\input{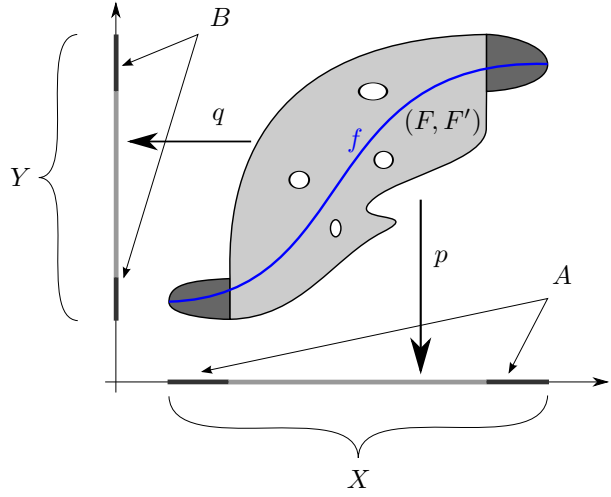}}
\caption{Notation used in the paper:
$(F, F')$ is a correspondence from $(X, A)$ to $(Y, B)$,
and $f \colon (X, A) \to (Y,B)$ is a continuous selector.
The maps $p \colon (F, F') \to (X, A)$
and $q \colon (F, F') \to (Y, B)$
are the canonical projections onto $X$ and $Y$, respectively.
Note that given the correspondence $(F, F')$, in general $p$ does 
not satisfy the acyclicity assumption of Theorem~\ref{thm:mapping},
but if $\ker p_* \subset \ker q_*$ then $f_*$ can be determined
from the projections (see Theorem~\ref{thm:homF}).}
\label{fig:graphFG}
\end{figure}


\section{Characterizing homomorphisms induced in homology by a correspondence}
\label{sec:main}

We begin by defining the notion of homological completeness
and homological consistency for correspondences,
and by establishing their basic properties.
Then we introduce the definition of the homomorphism
induced in homology by a correspondence,
and we prove our main result, Theorem~\ref{thm:homF}.


Our goal is to use a correspondence $(F,F')$
for which we can compute $(F, F')_*$ to determine
the induced map on homology of a continuous function $f$.
To do this requires and understanding of the information
that is lost or preserved by $\widetilde{q_*}$ and $\widetilde{p_*}^{-1}$.
We begin by considering $p_*$.

\begin{definition}
\label{def:complete}
A closed correspondence $(F, F')$ from $(X, A)$ to $(Y, B)$
is \emph{homologically complete} if the homomorphism $p_*$
induced in homology by the natural projection
$p \colon (F, F') \to (X, A)$ is an epimorphism.
\end{definition}

\begin{lemma}
\label{lem:complExt}
If $(F, F') \subset (G, G')$ are closed correspondences
from $(X, A)$ to $(Y, B)$,
and $p^F \colon (F, F') \to (X, A)$ and $p^G \colon (G, G') \to (X, A)$
are the natural projections,
then $\im p^F_* \subset \im p^G_*$.
\end{lemma}

\begin{proof}
Since obviously $p^F = p^G \circ \iota$,
where $\iota \colon (F, F') \to (G, G')$ is the inclusion,
the same holds true after applying the homology functor:
$p^F_* = p^G_* \circ \iota_*$.
The conclusion follows.
\end{proof}

\begin{proposition}
\label{prop:complExt}
Let $(F, F') \subset (G, G')$ be closed correspondences
from $(X, A)$ to $(Y, B)$.
If $(F, F')$ is homologically complete then so is $(G, G')$.
\end{proposition}

\begin{proof}
This result follows immediately from Lemma~\ref{lem:complExt}.
\end{proof}

\begin{remark}
\label{rem:complCont}
Since the projection from the graph of a continuous map
onto the domain of the map is a homeomorphism,
a continuous map is homologically complete.
\end{remark}

By Remark~\ref{rem:complCont} and Proposition~\ref{prop:complExt},
we immediately conclude the following:

\begin{corollary}
\label{cor:complCont}
If a closed correspondence $(F, F')$ has a continuous selector
then $(F, F')$ is homologically complete.
\end{corollary}

We now change our focus to $q_*$.

\begin{definition}
\label{def:consistent}
A closed correspondence $(F, F')$ from $(X, A)$ to $(Y, B)$
is \emph{homologically consistent} if $\ker p_* \subset \ker q_*$,
where $p \colon (F, F') \to (X, A)$ and $q \colon (F, F') \to (Y, B)$
are the natural projections.
\end{definition}

\begin{proposition}
\label{prop:consExt}
Let $(F, F') \subset (G, G')$ be closed correspondences
from $(X, A)$ to $(Y, B)$.
If $(G, G')$ is homologically consistent then so is $(F, F')$.
\end{proposition}

\begin{proof}
Let $p^F, p^G, q^F, q^G$ be the respective projections.
Consider the inclusion $\iota \colon (F, F') \to (G, G')$.
Note that $p^F = p^G \circ \iota$ and $q^F = q^G \circ \iota$.
Take any $c \in H_* (F, F')$ such that $p^F_* (c) = 0$.
Define $c' := \iota_* (c)$.
Then $p^G_* (c') = 0$.
Since $(G, G')$ is homologically consistent, $q^G_* (c') = 0$,
and thus $q^F_* (c) = 0$.
Since the choice of $c$ was arbitrary,
this reasoning shows that $(F, F')$ is homologically consistent.
\end{proof}

\begin{remark}
\label{rem:corrAcycl}
If $p_*$ is an isomorphism then $(F, F')$
is obviously homologically consistent.
In particular, if $(F, F')$ is a continuous map,
or $X$ and $Y$ are compact metric spaces
and $(F, F')$ is an acyclic upper semicontinuous multivalued map,
then $(F, F')$ is homologically consistent.
However, it is not enough for a multivalued map
to have a continuous selector
to be homologically consistent (see Example~\ref{ex:wrongMap}).
\end{remark}

\begin{corollary}
\label{cor:corrAcycl}
If $Y$ is compact and a closed correspondence $(F, F')$
from $(X, A)$ to $(Y, B)$ has an enlargement
that is an acyclic upper semicontinuous multivalued map
then $(F, F')$ is homologically consistent.
\end{corollary}

The following result justifies the usefulness of Definition~\ref{def:homF}.
See Figure~\ref{fig:graphFG} for a suggestive illustration of the set-up.

\begin{theorem}
\label{thm:homF}
Let $f \colon (X, A) \to (Y, B)$ be a continuous selector
of a closed correspondence $(F, F')$ from $(X, A)$ to $(Y, B)$.
If $(F, F')$ is homologically consistent then $f_* = (F, F')_*$.
\end{theorem}

\begin{proof}
By Corollary~\ref{cor:complCont},
the fact that $(F, F')$ has a continuous selector
implies that the correspondence $(F, F')$ is homologically complete,
and thus $\im p_* = H_* (X, A)$.
Since $(F, F')$ is homologically consistent, $\ker p_* \subset \ker q_*$,
and thus $q_* (\ker p_*) = 0$.
Therefore, $H_* (Y) / q_* (\ker p_*) \cong H_* (Y)$.

Let us prove that
$f_* = \widetilde{q_*} \circ (\widetilde{p_*})^{-1}$.
Let $j \colon (f, f|_A) \to (F, F')$ denote the inclusion.
Since $(f, f|_A)$ is homeomorphic with
$(X, A)$ by the natural projection $p \circ j$,
this homeomorphism induces the isomorphism
$(p \circ j)_* = p_* \circ j_*$ in homology.
The following diagram may help following the remaining part of the proof.
\[
\qquad\qquad
\begin{tikzcd}
& H_* (Y, B) \\
H_* (f, f|_A)
\arrow{r}{j_*} \arrow{ur}{(q \circ j)_*}
\arrow{dr}{\cong}[swap]{(p \circ j)_*} &
H_* (F, F')
\arrow{r}{\pi} \arrow{u}[swap]{q_*} \arrow{d}{p_*}&
H_* (F, F') / \ker p_*
\arrow{lu}[swap]{\widetilde{q_*}}
\arrow{ld}{\widetilde{p_*}}[swap]{\cong} \\
& H_* (X, A)
\arrow[end anchor=east,controls={+(6,0.5) and +(6,-0.5)}]{uu}[swap]{f_*}
\end{tikzcd}
\]
Let
$\pi \colon H_* (F, F') \to H_* (F, F') / \ker p_*$
denote the natural projection (which is an epimorphism).
Then obviously $p_* = \widetilde{p_*} \circ \pi$.
As a consequence,
\[
(p \circ j)_* = p_* \circ j_* = (\widetilde{p_*} \circ \pi) \circ j_* =
\widetilde{p_*} \circ (\pi \circ j_*),
\]
and thus
\[
(\widetilde{p_*})^{-1} = (\pi \circ j_*) \circ ((p \circ j)_*)^{-1}.
\]
The assumption that $\ker p_* \subset \ker q_*$
implies that $q_* = \widetilde{q_*} \circ \pi$
(because $\ker \pi = \ker p_*$).
Therefore,
\begin{multline*}
\widetilde{q_*} \circ (\widetilde{p_*})^{-1} =
\widetilde{q_*} \circ (\pi \circ j_*) \circ ((p \circ j)_*)^{-1} = \\
= ((\widetilde{q_*} \circ \pi) \circ j_*) \circ ((p \circ j)_*)^{-1} =
(q_* \circ j_*) \circ ((p \circ j)_*)^{-1} = \\
(q \circ j)_* \circ ((p \circ j)_*)^{-1},
\end{multline*}
and this coincides with $f_*$ (see Remark~\ref{rem:homFcont}),
which completes the proof.
\end{proof}


\section{Extensions of a correspondence}
\label{sec:extensions}

Applications such as reconstructing the graph of a map from time series
or using a finite sample, possibly with some noise,
to determine the graph of a map give rise to missing or perturbed data.
Thus the observed correspondence $(F, F')$ may differ
from the unknown true correspondence $(G, G')$.
We show that it is possible to retrieve some homological information
about $(G, G')$ from $(F, F')$.


\subsection{Homological extension and homologically consistent enlargment}
\label{sec:homExt}

A homological extension is meant to serve as a replacement
of a correspondence that extends the original one
at the homological level.
A homologically consistent enlargement is supposed to be
a correction of a correspondence
aimed at fixing the problem of missing data. More specifically:

\begin{definition}
\label{def:homMatching}
Let $(F, F')$ and $(G, G')$ be closed correspondences
from $(X, A)$ to $(Y, B)$.
Let $p^F$, $p^G$, $q^F$, and $q^G$ denote the respective projections.
We say that \emph{$(G, G')$ is a homological extension of $(F, F')$}
if the following conditions hold:
\begin{enumerate}
\item\label{homMatching1}
$\im p^F_* \subset \im p^G_*$,
which yields the natural inclusion homomorphism
\[
i \colon \im p^F_* \to \im p^G_*;
\]
\item\label{homMatching2}
$q^G_* (\ker p^G_*) \subset q^F_* (\ker p^F_*)$,
which provides the natural projection homomorphism
\[
j \colon H_* (Y, B) / q^G_* (\ker p^G_*)
\to H_* (Y, B) / q^F_* (\ker p^F_*);
\]
\item\label{homMatching3}
$(F, F')_* = j \circ (G, G')_* \circ i$.
\end{enumerate}
\end{definition}

The name ``extension'' is motivated by the fact
that the homomorphism $j \circ (G, G')_*$
restricted to the domain of $(F, F')_*$
actually equals $(F, F')_*$,
so indeed $(G, G')_*$ is an extension of $(F, F')_*$
up to an isomorphism.
Moreover, the correspondence $(G, G')$
carries possibly more homological information than $(F, F')$.
Indeed, $(F, F')_*$ can be recovered from $(G, G')_*$
by applying the maps $i$ and $j$.

\begin{remark}
\label{rem:homOrder}
The relation of being a homological extension
defines a partial order on closed correspondences from $(X, A)$ to $(Y, B)$.
\end{remark}

It is worth to mention that even if $(G, G')$
is a homological extension of $(F, F')$
then it is not necessary that $(F, F') \subset (G, G')$,
nor that $(G, G') \subset (F, F')$.
However, a converse holds true to certain extent:

\begin{proposition}
\label{prop:dataExt}
A homologically consistent enlargement is a homological extension.
\end{proposition}

\begin{proof}
Let $(G, G')$ be a homologically consistent enlargement of $(F, F')$.
We show that $(G, G')$ is a homological extension of $(F, F')$.
For that purpose, we show that conditions
\eqref{homMatching1}--\eqref{homMatching3}
of Definition~\ref{def:homMatching}
are satisfied.
\eqref{homMatching1} By Lemma~\ref{lem:complExt},
$\im p^F_* \subset \im p^G_*$.
\eqref{homMatching2} Since $(G, G')$ is homologically consistent,
$q^G_* (\ker p^G_*) = 0 \subset q^F_* (\ker p^F_*)$.
\eqref{homMatching1}
Suppose $F_* (x) = y$.
Let $\iota \colon (F, F') \to (G, G')$ be the inclusion map.
We show $( j \circ (G, G')_* \circ i) ( x ) = y$.
By definition of $(F, F')_* (x) = y$, there exists $c \in H_* (F, F')$
such that $p^F_* (c) = x$ and $y = q^F_* (c) + q^F_* (\ker p^F_*)$.
Consider $c' := \iota_* (c) \in H_* (G, G')$. By definition,
\[
(G, G')_*( p^G_* (c')) = q^G_* (c') + q^G_* (\ker p^G_*).
\]
Note that $p^F_* (c) = p^G_* (c') = x$.
Similarly, $q^F_* (c) = q^G_* (c')$.
Hence
\begin{multline*}
(j \circ (G, G')_* \circ i) (x) = (j \circ (G, G')_*) (x) =
j (q^G_* (c') + q^G_* (\ker p^G_*)) = \\
= j ( q^F_* (c) + q^G_* (\ker p^G_*)) =
q^F_* (c) + q^G_* (\ker p^G_*) + q^F_* (\ker q^F_*).
\end{multline*}
Since $0 = q^G_* (\ker p^G_*) \subset q^F_* (\ker q^F_*)$,
we conclude that $(j \circ (G, G')_* \circ i) (x) =
q^F_*(c) + q^F_* (\ker q^F_*) = y$, as desired.
\end{proof}

\begin{remark}
\label{rem:missing}
Proposition~\ref{prop:dataExt} shows that
we can learn partial or total information about the homomorphism
induced in homology by an unknown homologically consistent map
even when some data is missing.
In particular, if $(F, F')_*$ is nontrivial then the same holds true
for any of its homologically consistent enlargements.
\end{remark}

\begin{remark}
\label{rem:dataExt}
As an immediate consequence of Proposition~\ref{prop:consExt},
if there exists a homologically consistent enlargement
of a correspondence $(F, F')$
then $(F, F')$ is homologically consistent.
Note that by Proposition~\ref{prop:complExt},
homological completeness carries over to enlargements;
therefore, if $(F, F')$ is homologically complete
then so are any of its homologically consistent enlargements.
\end{remark}

The concept of homologically consistent enlargements
allows us to give the following generalization of Theorem~\ref{thm:homF}:

\begin{theorem} \label{thm:combinedtheorem}
Let $f \colon (X, A) \to (Y, B)$ be a continuous selector
of a homologically consistent enlargement $(G, G')$
of a closed correspondence $(F, F')$ from $(X, A)$ to $(Y, B)$.
If $(F, F')$ is homologically complete then $f_* = (F, F')_*$.
\end{theorem}

\begin{proof}
By Theorem \ref{thm:homF},
we have $f_* = (G, G')_*$. Via Proposition~\ref{prop:dataExt},
$(G, G')$ is a homological extension of $(F, F')$.
By homological completeness of $(F,F')$,
the map $i$ of Definition~\ref{def:homMatching}
is identity. By Proposition~\ref{prop:consExt},
$(F,F')$ is homologically consistent,
and the map $j$ of Definition \ref{def:homMatching}
is identity as well. It follows from property (3)
of Definition~\ref{def:homMatching} that $(F, F')_* = (G,G')* = f_*$.
\end{proof}


\subsection{Existence of an acyclic enlargement}
\label{sec:acyclExt}

We prove that if the images of a closed correspondence
are small enough then the correspondence has an acyclic enlargement,
that is, an enlargement which is an acyclic multivalued map;
see Theorem~\ref{thm:extAcycl} below.
The required size of the images is provided explicitly
if we know a Lebesgue number of a cover of $Y$
that satisfies certain conditions.
In particular, this result provides an explicit sufficient condition
on how tight an approximation of an unknown continuous map must be
in order to be sure that it provides meaningful homological information.

\begin{definition}
\label{def:closedCover}
A collection $\cU$ of closed subsets of a metric space $Y$
is called a \emph{closed cover} of $Y$ if the interiors
of the sets in $\cU$ form an open cover of $Y$.
\end{definition}

\begin{definition}
\label{def:goodCover}
A closed cover $\cU$ of a metric space $Y$
is called a \emph{good closed cover} of $Y$
if the intersection of any finite collection
of elements of $\cU$ is either empty or acyclic.
\end{definition}

\begin{remark}
\label{rem:finGoodCover}
If $\cU$ is a good closed cover of $Y$ then any sub-cover $\cV$ of $Y$
(that is, a closed cover $\cV$ of $Y$ such that $\cV \subset \cU$)
is a good closed cover of $Y$, too.
In particular, if $Y$ is compact and has a good closed cover $\cU$
then it has a finite good closed cover.
(The finite sub-cover of $\cU$ that exists
by the compactness of $Y$ is a good choice.)
\end{remark}

A minor modification of the proof conducted in \cite[pg. 42--43]{BT95}
for the case of an open cover, allows us to claim the following.

\begin{proposition}
\label{prop:strAcycl}
Every second countable Hausdorff manifold has a good closed cover.
\end{proposition}

\begin{definition}
\label{def:inscribed}
We say that a relation $F \subset X \times Y$
is \emph{inscribed} into a closed cover $\cU$ of $Y$
if for every $x \in X$ there exists $U \in \cU$ such that
$F (x) \subset \interior U$.
We say that a correspondence $(F, F')$ is \emph{inscribed}
into a closed cover $\cU$ of $Y$ if $F$ is inscribed into $\cU$.
(Note that then obviously $F' \subset F$ is also inscribed into $\cU$.)
\end{definition}

\begin{remark}
\label{rem:Lebesgue}
If $Y$ is a compact metric space
then it follows from the Lebesgue's number lemma
that if the diameters of $F (x)$ are bounded
by a Lebesgue's number of the open cover
$\cU' := \{\interior U : U \in \cU\}$
then $F$ is inscribed into $\cU$.
\end{remark}

\begin{theorem}
\label{thm:extAcycl}
Let $X$ and $Y$ be compact metric spaces.
Assume that $Y$ has a finite good closed cover $\cU$.
Let $A \subset X$ be closed,
and take $B := \bigcup \cU'$ for some $\cU' \subset \cU$.
Let $(F, F')$ be an upper semicontinuous multivalued map
from $(X, A)$ to $(Y, B)$ inscribed into $\cU$.
Then $(F, F')$ has an acyclic enlargement.
\end{theorem}

\begin{proof}
Denote the elements of $\cU$ by $U_1, \ldots, U_r$.
For each $i \in \{1, \ldots, r\}$,
define $W_i := \{x \in X : F (x) \subset \interior U_i\}$.
Since $F$ is inscribed into $\cU$,
the collection $\{W_1, \ldots, W_r\}$ is a cover of $X$.
By semicontinuity, the sets $W_i$ are open.
Define $G (x) := \bigcap \{U_i \in \cU : x \in W_i \}$
and $G' := G|_A$.
Obviously, $(F, F') \subset (G, G')$.
Moreover, $(G, G')$ is acyclic, because $\cU$ is a good closed cover.
\end{proof}

\begin{remark}
Let $(G, G')$ be an acyclic enlargement of $(F, F')$
in Theorem~\ref{thm:extAcycl}.
Then by Vietoris-Begle Mapping Theorem~\ref{thm:mapping},
the projection $p_G \colon (G, G') \to (X, A)$
induces an isomorphism in homology.
Note that it follows that the correspondence $(G, G')$
is homologically consistent.
\end{remark}


\section{Combinatorial maps and examples}
\label{sec:comb}

When using computational methods
for the determination of the homomorphism induced in homology,
one may represent a correspondence
by means of a finite combinatorial structure.
We provide the essential definitions and then provide a series of examples.


\subsection{Preliminaries}
\label{sec:combPrelim}

A \emph{grid} $\cX$ for a compact set $X$
is a finite collection of regular compact subsets of $X$
with disjoint interiors such that
$\bigcup \cX = \bigcup_{\xi \in \cX} \xi = X$.
The \emph{geometric realization} of a set $\cU \subset \cX$
is $|\cU| := \bigcup \cU = \bigcup_{\xi \in \cU} \xi$.

Let $\cX$ be a grid in $X$ and let $\cY$ be a grid in $Y$.
A \emph{combinatorial map} is
a multivalued map $\cF \colon \cX \mvmap \cY$,
that is, a set-valued map $\cF \colon \cX \to 2^{\cY}$,
such that $\cF (\xi) \neq \emptyset$ for all $\xi \in \cX$.
If $\cA \subset \cX$ and $\cB \subset \cY$
and $\cF (\cA) := \bigcup \{\cF (\xi) : \xi \in \cA\} \subset \cB$
then we denote $\cF \colon (\cX, \cA) \mvmap (\cY, \cB)$.

The \emph{geometric realization}
of a multivalued map $\cF \colon \cX \mvmap \cY$
is the relation $F \subset X \times Y$
defined as follows: $(x, y) \in F$ if there exist
$\xi_X \in \cX$ and $\xi_Y \in \cY$
such that $x \in \xi_X$, $y \in \xi_Y$, and $\xi_Y \in \cF (\xi_X)$.
The geometric realization of a multivalued map
$\cF \colon (\cX, \cA) \mvmap (\cY, \cB)$ is the correspondence $(F, F')$
such that $F$ is the geometric realization of $\cF \colon \cX \mvmap \cY$
and $F'$ is the geometric realization of the multivalued map
$\cF' \colon \cA \mvmap \cB$,
where $\cF' (\xi) := \cF (\xi)$ for all $\xi \in \cA$.
It is easy to see that the geometric realization
of a combinatorial map is an upper semicontinuous multivalued map.

\begin{definition}
A combinatorial map $\cF \colon (\cX, \cA) \mvmap (\cY, \cB)$
is a \emph{representation} of a continuous map $f \colon (X,A) \to (Y,B)$,
where $X = |\cX|$, $A = |\cA|$, $Y = |\cY|$, $B = |\cB|$,
if $f (\xi) \subset |\cF (\xi)|$ for all $\xi \in \cX$.
\end{definition}

We remark that in many cases of practical interest,
if an analytic description of $f$ is given explicitly or implicitly
then a representation of $f$ is effectively computable.
Furthermore, this representation can be chosen to approximate $f$
sufficiently closely to ensure that the assumptions
of Theorem~\ref{thm:extAcycl} are satisfied.
A variety of efficient tools for this are provided by CAPD \cite{capd}.

To construct representations in the context of data
we introduce the following concept.

\begin{definition}
\label{def:sampling}
Given a continuous map $f\colon (X,A)\to (Y,B)$,
any finite subset $S$ of the graph of $f$
is called \emph{sampling data for $f$}.
Given grid representations $(\cX, \cA)$ and $(\cY, \cB)$
for $(X,A)$ and $(Y, B)$, respectively, define
\emph{the combinatorial map
$\cF^S \colon (\cX,\cA) \mvmap (\cY,\cB)$ induced by the samples $S$}
as follows: $\eta \in \cF^S(\xi)$ if and only if
$x \in \vert \xi \vert$ and $f(x)\in\vert\eta\vert$
for some $(x,f(x))\in S$.
\end{definition}


Assume that a combinatorial representation
$\cF \colon (\cX, \cA) \mvmap (\cY, \cB)$
of an unknown continuous map $f$ is given
and we are interested in the computation of $f_*$.
Let $(F, F')$ be the geometric realization of $\cF$.
Then $\cF$ is said to be {\em homologically complete}
(respectively, {\em consistent}) whenever $(F, F')$
is homologically complete (respectively, consistent).
A combinatorial map $\cG$
\emph{is a homologically consistent enlargement of $\cF$}
if the geometric realization of $\cG$
is a homologically consistent enlargement
of the geometric realization of $\cF$.
Consider the projections $p \colon (F, F') \to (X, A)$
and $q \colon (F, F') \to (Y, B)$.
If $\cF$ is both homologically complete and homologically consistent,
then we define
\begin{equation}
\label{def:inducedMapOfComb}
\cF_* := \widetilde{q_*} \circ (\widetilde{p_*})^{-1},
\end{equation} 
where $\widetilde{p_*}$ and $\widetilde{q_*}$
are the homomorphisms induced by $p_*$ and $q_*$,
respectively, on the quotient space $H_* (F, F') / \ker p_*$.

\begin{remark}
\label{rem:homExtForComb}
The results in Section \ref{sec:homExt} transparently carry over
to combinatorial maps by applying them to their geometric realizations.
\end{remark}

\begin{remark}
\label{rem:hom}
Observe that \eqref{def:inducedMapOfComb} is independent of the method
chosen to compute $p_*$ and $q_*$.
A variety of different techniques to perform these computations
can be found in \cite{HMMN13,MMP05,PR14}.
\end{remark}

Returning to the setting of data, we have the following result.

\begin{proposition}
\label{prop:sampling}
Let $f \colon (X, A) \to (Y, B)$ be a continuous map
admitting a homogically consistent representation
$\cF \colon (\cX,\cA) \mvmap (\cY,\cB)$.
Let $S$ be some sampling data for $f$.
Let $\cF^S$ be the combinatorial map induced by the samples $S$.
Then $\cF^S$ is homologically consistent.
Moreover, if $\cF^S$ is homologically complete then $f_* = \cF^S_*$.
\end{proposition}

\begin{proof}
By Definition~\ref{def:sampling}, $\cF^S \subset \cF$.
Then $\cF^S$ is homologically consistent
by Proposition~\ref{prop:consExt}.
The remaining result is an immediate consequence
of Theorem~\ref{thm:combinedtheorem}.
\end{proof}

\begin{remark}
In fact, we can do somewhat better than Proposition~\ref{prop:sampling}
and tolerate \emph{noisy samples}. A sufficient condition is
that there exists a homologically consistent correspondence
containing both $f$ and the noisy samples.
Given known bounds on the noise, it may be possible to give
a general result of this form using ideas along the lines
of Section~\ref{sec:acyclExt}. We do not attempt this here,
but please see the discussion in Example~\ref{example:noisysample}.
\end{remark}


\subsection{Applications and Examples}
\label{sec:reconstr}

In what follows, we provide three specific examples which illustrate
the importance of homological consistency and homological completeness.
We give an example satisfying homological completeness
and homological consistency, an example where homological consistency fails,
and an example where homological completeness fails.
To make the computations discussed in Examples \ref{example:noisysample}
and \ref{ex:henon} reproducible, the software used for these computations
is available at~\cite{chompdata}.

\begin{example}
\label{example:noisysample}
Our first example illustrates a successful application
of Theorem~\ref{thm:combinedtheorem}.
Consider a sampling from the a double winding map
\[
f( \cos \theta, \sin \theta ) = ( \cos 2\theta, \sin 2\theta ).
\]
Rather than using~Proposition~\ref{prop:sampling},
we want to demonstrate the full strength
of Theorem~\ref{thm:combinedtheorem}.
To this end we generated a combinatorial map
$\mathcal{F}\colon \cX\mvmap\cX$ by computing samples of $f$
but adding a small level of random noise.
The noise results in $f$ not being a continuous selector of $\mathcal{F}$.
Nevertheless, as we demonstrate the conditions
of Theorem~\ref{thm:combinedtheorem} are satisfied
and thus $\mathcal{F}$ can be used to obtain $f_*$.

To generate our example we proceed as follows.
We divide the plane region $[-2,2]^2$ into $256 \times 256$ grid elements
and choose $\cX$ to be those grid elements within $.01$ distance
(under the sup-norm) to the unit circle.
We generate $3{,}000$ samples on the unit circle of the plane,
evaluate $f$ on each sample point, add bounded noise
uniformly distributed in $[-.1,.1]^2$,
and radially project the result to the unit circle.
From these noisy samples of $f$ we constructed a combinatorial map $\cF$.
We check computationally that (i) $\mathcal{F}$ does not contain $f$
and that (ii) $\mathcal{F}$ is homologically complete.
The fact that $\mathcal{F}$ admits a homologically consistent enlargement
containing $f$ follows from the results of Section \ref{sec:acyclExt},
an appropriate choice of good cover,
and the observation that the noise was small.
We omit the details, but remark that this implies
that Theorem~\ref{thm:combinedtheorem} is applicable.

As a check we compute the homomorphism induced in homology
by the closed correspondence $(F,F')$ corresponding to $\mathcal{F}$
and obtain $H_*(F, F') \cong (\mathbb{Z}^2, \mathbb{Z}^{29}, 0, \ldots)$.
Computing the induced map on homology of the projections $p$ and $q$,
we find that $p_*(g) = q_*(g) = 0$ for all but one $0$-cycle basis element,
where we had $p_*(g_0) = q_*(g_0)$ and all but one $1$-cycle basis element,
where we found $q_*(g_1) = 2 p_*(g_1)$. Hence $\mathcal{F}_* = f_*$,
the induced map on homology of a double winding map,
in accordance with Theorem~\ref{thm:combinedtheorem}.
\end{example}

\begin{figure}[htbp]
\def\svgwidth{0.35\textwidth}
\centerline{\input{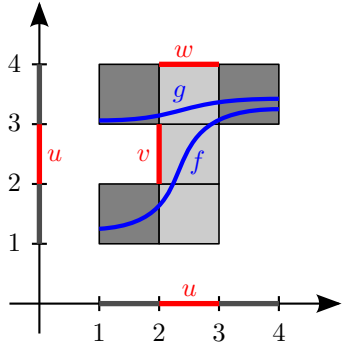}}
\caption{A sample representation $\cF$ of two continuous maps
$f$ and $g$ such that $f_* \neq g_*$,
described in Example~\ref{ex:wrongMap}.
The homology generators discussed in the text
correspond to the segments labeled $u$, $v$ and $w$.
Note that
$\ker p_* = \langle v \rangle \not\subset \langle w \rangle = \ker q_*$.}
\label{fig:wrongMap}
\end{figure}

\begin{example} 
\label{ex:wrongMap}
Our second example illustrates a very simple situation
that shows that if $(F, F')$ is not homologically consistent
(that is, $\ker p_* \not\subset \ker q_*$),
then this may lead to ambiguity in determining $f_*$ from $\cF$.
Consider the following objects (see Figure~\ref{fig:wrongMap}):
\begin{eqnarray*}
&& \cX := \cY := \{[1,2], [2,3], [3,4]\}, \\
&& \cA := \cB := \{[1,2], [3,4]\}, \\
&& \cF ([1,2]) := \{[1,2], [3,4]\}, \\
&& \cF ([2, 3]) := \{[1,2], [2,3], [3,4]\}, \\
&& \cF ([3, 4]) := \{[3,4]\}.
\end{eqnarray*}
Let $(F, F')$ be the geometric realization
of $\cF \colon (\cX, \cA) \mvmap (\cY, \cB)$.
Then $H_1 (X, A) \cong R$ and $H_1 (F, F') \cong R \oplus R$,
and the homology of these spaces
at the levels different from $1$ is trivial.
Let $u$ be a generator of $H_1 (X, A) = H_1 (Y, B)$
corresponding to the segment $[2, 3]$.
Choose generators $v$, $w$ of $H_1 (F, F')$
so that $v$ corresponds to $\{2\} \times [2, 3]$
and $w$ corresponds to $[2, 3] \times \{4\}$
(see Figure~\ref{fig:wrongMap}).
Then $p_1 (v) = 0$ and $p_1 (w) = u$,
and thus $\ker p_1 = \langle v \rangle$.
On the other hand, $q_1 (v) = u$ and $q_1 (w) = 0$,
so $\ker q_1 = \langle w \rangle$.
Obviously, $\ker p_* \not\subset \ker q_*$.
Indeed, it is possible to show two continuous maps
such that $f_1 (u) = u$ and $g_1 (u) = 0$;
in particular, $f_*$ is a nontrivial homomorphism,
and $g_*$ is trivial.
Graphs of such sample maps
are shown in Figure~\ref{fig:wrongMap}.
\end{example}

\begin{example}
\label{ex:henon}
For our third example, we consider a combinatorial map
which is homologically consistent but not homologically complete.
We analyze time series generated by the H\'enon map 
\begin{equation}
\label{eq:henon}
 (x,y) \mapsto (1 - ax^2 + y, bx)
\end{equation}
at the classical parameter values $a = 1.4$ and $b = 0.3$.
We produce a neighborhood of the attractor from a time series
generated of the form $(x_n)_{n=100}^{N}$, $N=100{,}000$
where $x_0 = (0,0)$. We construct a grid by uniformly dividing
the planar region $[-2,2]^2$ into $256 \times 256$ grid elements.
We set $\cX$ to be the set of grid elements that contain at least one
of the points from the time-series and set $\cA =\emptyset$.
See Figure \ref{fig:HenonGrid}.

\begin{figure}[htbp]
{%
\setlength{\fboxsep}{.01pt}%
\setlength{\fboxrule}{.1pt}%
\fbox{\includegraphics[width=\textwidth,height=0.3030303\textwidth]{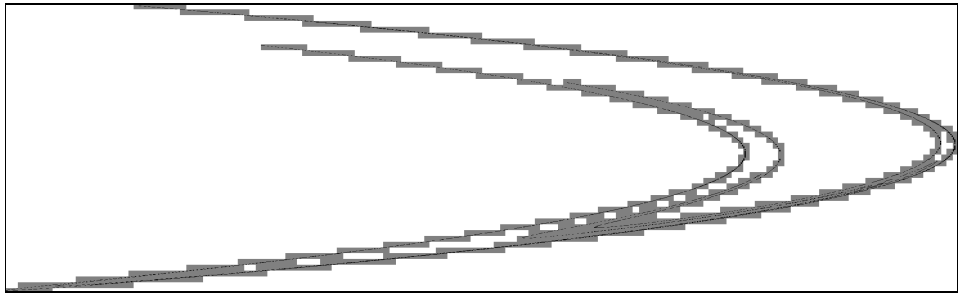}}%
}%
\caption{\label{fig:HenonGrid}%
The grid $\cX$ computed in the planar region $[-2,2]^2$
to outer approximate the time series points
approaching the H\'enon attractor.
The bounds of the image are
$[-1.29688, 1.28125] \times [-0.390625, 0.390625]$.
The time series points are plotted in black;
the grid elements of $\cX$ are gray.
Note the 1-cycles produced by the over-approximation.}
\end{figure}

The combinatorial map $\cF$ is defined by the property
that $\eta \in \cF(\xi)$ whenever $x_i \in \vert\xi\vert$
and $x_{i+1} \in \vert\eta\vert$ for some $i \geq 100$.
We compute the homomorphism induced in homology
by the correspondence $(F, F')$ associated with $\cF$,
the homology of the planar set $X$ corresponding to $\cX$,
and the induced maps on homology $p_*$ and $q_*$
for the projections $p, q$ from $(F,F')$ to $(X,A)$.
(Note: $F'=A=\emptyset$.) We find that
$H_*(F, F') = (\mathbb{Z},\mathbb{Z}^{21}, 0, 0, \cdots)$
and $H_*(X, A) = (\mathbb{Z}, \mathbb{Z}^{18}, 0, 0, \cdots)$.
The specific entries of the matrices $p_*$ and $q_*$
is of limited importance so we only report the following information:
(i)~we obtain homological consistency,
and (ii)~on the first level $p_*$ is rank deficient;
in particular $\mbox{rank }p_1 = 15 < 18$.
Since we see that $p_1$ is not surjective,
we conclude that $\cF$ is not homologically complete.

It is worth exploring this example further. 
As is indicated in Henon's original work~\cite{henon}
at the chosen parameter values \eqref{eq:henon} is invertible
and there exists a simply connected polygonal trapping region.
Thus, at first glance it is reasonable to expect
that the neighborhood of the attractor characterized by $\cX$
should be simply connected.
In particular, the $18$ 1-cycles of $H_1(X,A)$
(these are visible in Figure \ref{fig:HenonGrid} as interior holes
of the gray region) are `artifacts' due to the outer approximation
via the cubical grid $\cX$.
The map $p_*$ is not surjective since for at least a few
of these artifacts there is no corresponding cycle
in the complex corresponding to the combinatorial map $\cF$.
We can intuitively understand how this happens.
An artifact occurs in $H_*(X,A)$ when two disconnected parts
of the H\'enon attractor are present in the same grid element
and allow a spurious cycle to form.
However if those disconnected parts of the attractor
map to sufficiently separated points then we will not see
a corresponding spurious cycle formed in the combinatorial map.
We might think to eliminate these artifacts by refining our grid,
but due to the fractal-like nature of the H\'enon attractor one expects
to compute artifactual homology at every level of resolution.
Thus a lack of homological completeness is a general phenomenon
that should be expected when analyzing time series on strange attractors.

An interesting idea is to use the induced maps on homology of projections
from the combinatorial map to clean up at least some of the artifacts
in the homology of an outer approximation of an attractor.
Once we do this, we replace $H_*(X,A)$ with $\mbox{im } p_*$
and recover homological completeness. In the current example,
doing this produces a map $q_* \circ p_*^{-1}$.
We checked and found that this map was nilpotent.
This agrees with what we expect, as Conley theory tells us
that the Conley index of a simply connected attractor on the 1st level
ought to be shift equivalent to the zero matrix.
It is not clear what the conditions must be
in order to rigorously justify this procedure;
we leave the precise development of this idea to future work.
\end{example}


\section*{Acknowledgements}
\label{sec:ack}

The authors gratefully acknowledge the support
of the Lorenz Center which provided an opportunity for us
to discuss in depth the work of this paper.
Research leading to these results has received funding from
Fundo Europeu de Desenvolvimento Regional (FEDER)
through COMPETE---Programa Operacional Factores de Competitividade (POFC)
and from the Portuguese national funds
through Funda\c{c}\~{a}o para a Ci\^{e}ncia e a Tecnologia (FCT)
in the framework of the research project FCOMP-01-0124-FEDER-010645
(ref.\ FCT PTDC/MAT/098871/2008),
as well as from the People Programme (Marie Curie Actions)
of the European Union's Seventh Framework Programme (FP7/2007-2013)
under REA grant agreement no.~622033 (supporting PP).
The work of KM and SH has been partially supported
by NSF grants NSF-DMS-0835621, 0915019, 1125174, 1248071,
and contracts from AFOSR and DARPA.
The work of HK was supported by Grant-in-Aid
for Scientific Research (No.\ 25287029),
Ministry of Education, Science, Technology, Culture and Sports, Japan.


\bibliographystyle{amsplain}
\bibliography{mapbib}


\end{document}